\documentclass[12pt]{amsart}
\usepackage{amsthm,amsmath,stmaryrd,bbm,geometry,color}
\usepackage{amssymb}
\usepackage[english]{babel}
 \usepackage[utf8]{inputenc}
\usepackage{graphicx}
\usepackage{verbatim}
\usepackage{enumitem}
\usepackage{mathtools}
\usepackage[titletoc]{appendix}
\usepackage{bm}
\usepackage{textcomp}
\usepackage{amsaddr}

\usepackage[dvipsnames]{xcolor}
\usepackage[hyperindex=true,frenchlinks=true,colorlinks=true,
citecolor=Mahogany,linkcolor=Red,urlcolor=Tan,linktocpage]{hyperref}

\usepackage{tikz}
\usetikzlibrary{shapes}
\usetikzlibrary{fit}
\usetikzlibrary{decorations.pathmorphing}

\allowdisplaybreaks

\title[Functional central limit theorems for $U$-statistics of $\beta$-mixing data]{Functional central limit theorems for $U$-statistics of $\beta$-mixing data}
\author{Davide Giraudo}
\address{Institut de Recherche Mathématique Avancée UMR 7501,
Université de
Strasbourg and CNRS 7 rue René Descartes 67000 Strasbourg, France }
\email{dgiraudo@unistra.fr}
\date{\today}

\numberwithin{equation}{section}
\renewcommand{\leq}{\leqslant}
\renewcommand{\geq}{\geqslant}

\newtheorem{Theorem}{Theorem}[section]
\newtheorem{Proposition}[Theorem]{Proposition}
\newtheorem{Lemma}[Theorem]{Lemma}
\newtheorem{Definition}[Theorem]{Definition}

\theoremstyle{remark}

\tikzstyle{Vertex}=[circle,draw=LimeGreen!80,fill=LimeGreen!8,
inner sep=1pt,minimum size=2mm,line width=1pt,font=\scriptsize]
\tikzstyle{Node}=[Vertex,draw=RoyalBlue!80,fill=RoyalBlue!8,inner
sep=1.5pt]
\tikzstyle{Leaf}=[rectangle,draw=Black!70,fill=Black!16,
inner sep=0pt,minimum size=1mm,line width=1.25pt]
\tikzstyle{Edge}=[Maroon!80,cap=round,line width=1pt]
\tikzstyle{Mark1}=[draw=BrickRed!80,fill=BrickRed!8]
\tikzstyle{Mark2}=[draw=BurntOrange!80,fill=BurntOrange!8]
\tikzstyle{EdgeRew}=[->,RedOrange!80,cap=round,thick]

\newcommand{\Aca}{\mathcal{A}}
\newcommand{\Bca}{\mathcal{B}}

\newcommand{\Fca}{\mathcal{F}}

\newcommand{\Hca}{\mathcal{H}}

\newcommand{\Nca}{\mathcal{N}}

\newcommand{\Sca}{\mathcal{S}}

\newcommand{\intent}[1]{\llbracket #1\rrbracket}

\newcommand \ens[1]{\left\{ #1\right\}}

\newcommand \R{\mathbb R}

\newcommand \N{\mathbb N}
\newcommand \PP{\mathbb P}

\newcommand{\E}[1]{\mathbb E\left[#1\right]}

\newcommand \Z{\mathbb Z}

\newcommand \abs[1]{\left|#1\right|}
\newcommand \eps{\varepsilon}

\newcommand{\pr}[1]{\left(#1\right)}
\newcommand{\norm}[1]{\left\lVert #1 \right\rVert}
\newcommand{\gr}[1]{\bm{#1}}

\newcommand{\gri}{\gr{i}}

\newcommand{\inc}{\operatorname{Inc}}

\newcommand{\ent}[1]{\left\lfloor #1\right\rfloor}

\begin{document}
 
\begin{abstract}
    We investigate the convergence of partial sum processes based on a strictly stationary $\beta$-mixing sequence of random variables. The convergence in the space of continuous function as well as in H\"older spaces is considered. The conditions are close to optimality.
\end{abstract}

\maketitle
\section{Introduction}
Given a sequence of random variables $\pr{X_i}_{i\geq 1}$ taking values in a measurable space $\pr{S,\Sca}$ and a kernel $h\colon S^m\to\R$, the $U$-statistic of order $m\geq 2$, introduced by Hoeffding \cite{MR26294}, is defined as 
\begin{equation*}
    U_{m,n}\pr{h}=\sum_{\gri\in\inc^m_n}h\pr{X_{\gri}},
\end{equation*}
where $\inc^m_n=\ens{\gri=\pr{i_1,\dots,i_m},1\leq i_1<\dots<i_m\leq n}$ and for $\gri=\pr{i_1,\dots,i_m}\in \inc^m_n$, 
$X_{\gri}=\pr{X_{i_1},\dots,X_{i_m}}$. 
In the sequel, we will assume that $h$ is symmetric, that is, for each bijection $\sigma\colon\intent{1,m}\to \intent{1,m}$ and each $x_1,\dots,x_m\in S$, $h\pr{x_{\sigma\pr{1}},\dots,x_{\sigma\pr{m}}}=h\pr{x_1,\dots,x_m}$, where for integers $a$ and $b$ such that $a\leq b$ $\intent{a,b}:=\ens{k\in\N,a\leq k\leq b}$.
When $\pr{X_i}_{i\geq 1}$ is independent identically distributed (i.i.d.) and $h\pr{X_1,\dots,X_m}\in\mathbb L^1$, $U_{m,n}\pr{h}/\binom{n}{m}$ is an unbiaised estimator of $\E{h\pr{X_1,\dots,X_m}}$.

The aim of this paper is to complement the results on the asymptotic behavior of $\pr{U_{m,n}\pr{h}}_{n\geq m}$  by 
the study of the convergence of the partial sum process 
defined by 
\begin{equation}\label{eq:Uprocess}
W_{m,n}\pr{h,t}=\begin{cases}
\sum_{\gri\in\inc_m^{j}}h\pr{X_{\gri}},& \quad t=\frac jn\\
\mbox{linear interpolation }&\mbox{on }\pr{\frac jn,\frac{j+1}n}, j\in\intent{0,n-1}.
\end{cases}  
\end{equation}
The map $t\mapsto W_{m,n}\pr{h,t}$ is continuous and 
the process $\pr{W_{m,n}\pr{h,t}}_{0\leq t\leq 1}$ contains the value of the $U$-statistic based on each subsample $X_1,\dots,X_k$, $k\in\intent{1,n}$.
Some statistical test can be performed using a continuous functional on the space $C\pr{[0,1]}$, using for instance 
the supremum norm. Moreover, since the limit process belongs to the space of Hölder continuous functions as well as the maps $t\mapsto W_{m,n}\pr{h,t}$, it makes sense to study the convergence of the U-process defined by \eqref{eq:Uprocess} in these function spaces. Beyond the mathematical curiosity, test statistics based on Hölder norm allows to detect a change of parameter over a short subset of $\intent{1,n}$, whose size depends on the modulus of regularity of the considered Hölder space.

\begin{Definition}
We say that a symmetric kernel $h\colon S^m\to \R$ is canonical with respect to a strictly stationary sequence $\pr{X_i}_{i\in\Z}$ if for each $y_1,\dots,y_{m-1}\in S$,
\[
\int h\pr{x,y_1,\dots,y_{m-1}}d\PP_{X_0}\pr{y_1}\dots 
d\PP_{X_0}\pr{y_{m-1}}=0.
\]
\end{Definition}
Hoeffding's decomposition (see \cite{MR26294}, section 5) allows to express the original $U$-statistic in terms of a sum of $U$-statistics of lower order having a canonical kernel. More precisely, 
define 
\begin{equation}\label{eq:definition_de_pi_km}
   \pr{\pi_{k,m}\pr{h}}\pr{x_1,\dots,x_k}:=
\pr{\delta_{x_1}-\PP_{X_0}}\dots \pr{\delta_{x_k}-\PP_{X_0}}\PP_{X_0}^{\otimes \pr{m-k}}h, 
\end{equation}
where for signed measures $Q_1,\dots Q_m$, 
\[Q_1\dots Q_m h =
\int h\pr{x_1,\dots,x_m}dQ_1\pr{x_1}\dots dQ_m\pr{x_m}
\]
and $\PP_{X_0}^{\otimes j}$ denotes the product measure $\PP_{X_0}$ taken $j$ times.
The following decomposition holds: 
\begin{equation}\label{eq:Hoeffding}
\frac 1{\binom nm}    U_{m,n}\pr{h}
=\sum_{c=0}^{m}\frac{\binom{m}{c}}{\binom{n}{c}}
U_{c,n}\pr{\pi_{c,m}\pr{h}},
\end{equation}
where each kernel $\pi_{c,m}\pr{h}, c\in\intent{1,m}$ is canonical.
After centering, this gives 
\begin{multline}\label{eq:Hoeffding_centered}
    U_{m,n}\pr{h}-\E{U_{m,n}\pr{h}}\\ =
    \frac{m}{n} \binom nm \pr{U_{1,n}\pr{\pi_{1,m}\pr{h}} 
    -\E{U_{1,n}\pr{\pi_{1,m}\pr{h}} }}+
    \sum_{c=2}^{m} \binom nm  \frac{\binom{m}{c}}{\binom{n}{c}}
U_{c,n}\pr{\pi_{c,m}\pr{h}}.
\end{multline}
Usually, the convergence is carried by the first term of the right hand side of \eqref{eq:Hoeffding_centered} (which will be referred as "linear part") while the contribution of the other terms is negligible.
When the sequence $\pr{X_i}_{i\in \Z}$ is i.i.d., 
martingale properties show that the contribution of the term $\sum_{c=2}^{m} \binom nm  \frac{\binom{m}{c}}{\binom{n}{c}}
U_{c,n}\pr{\pi_{c,m}\pr{h}}$ is negligible. When $\pr{X_i}_{i\in \Z}$ is strictly stationary but not independent, the increments in each term of summand involved in $U_{c,n}\pr{\pi_{c,m}\pr{h}}$ are in general not orthogonal. In this paper, we are interested in $\beta$-mixing sequences which are defined as follows.  

Let $\pr{\Omega,\Fca,\PP}$ be a probability space. 
The $\alpha$-mixing and $\beta$-mixing coefficients between two sub-$\sigma$-algebras 
$\Aca$ and $\Bca$ of $\Fca$ are defined respectively by 
\begin{align*}
 \alpha\pr{\Aca,\Bca}&=\sup\ens{
 \abs{\PP\pr{A\cap B}-\PP\pr{A}\PP\pr{B}},A\in\Aca, B\in \Bca
 }\\
 \beta\pr{\Aca,\Bca}&=\frac 12\sup\ens{
 \sum_{i=1}^I\sum_{j=1}^J\abs{\PP\pr{A_i\cap B_j}-
 \PP\pr{A_i}\PP\pr{B_j}}},
\end{align*}
where the supremum runs over all the partitions 
$\pr{A_i}_{i=1}^I$ and $\pr{B_j}_{j=1}^J$ of $\Omega$ of 
elements of $\Aca$ and $\Bca$ respectively.
Given a strictly stationary sequence $\pr{X_i}_{i\geq 1}$, we associate its 
sequences of $\alpha$ and $\beta$-mixing coefficients
by letting 
\begin{equation*}
\alpha\pr{k}:=\sup_{\ell\geq 1}
\alpha\pr{\Fca_1^\ell,\Fca_{\ell+k}^{\infty}}, \quad 
\beta\pr{k}:=\sup_{\ell\geq 1}
\beta\pr{\Fca_1^\ell,\Fca_{\ell+k}^{\infty}},
\end{equation*}
where $\Fca_u^v$, $1\leq u\leq v\leq +\infty$ is 
the $\sigma$-algebra generated by the random variables 
$X_i$, $u\leq i\leq v$ ($u\leq i$ for $v=\infty$). A sequence $\pr{X_i}_{i\geq 1}$ is said to be absolutely regular if $\lim_{k\to +\infty}\beta\pr{k}=0$.
 
The main goals of the papers are the following: give sufficient condition on integrability of the random variables $h\pr{X_{\gri}}$ and on the sequences 
$\pr{\alpha\pr{k}}_{k\geq 1}$ and $\pr{\beta\pr{k}}_{k\geq 1}$ for the convergence of :
\begin{enumerate}
    \item an appropriately centered and normalized version of $U_{m,n}\pr{h}$ to a normal distribution,
    \item an appropriately centered and normalized version of the process $\pr{W_{m,n}\pr{h,t},t\in [0,1]}$ to a Brownian motion in $C\pr{[0,1]}$ 
    and
    \item an appropriately centered and normalized version of the process $\pr{W_{m,n}\pr{h,t},t\in [0,1]}$ to a Brownian motion in the space of H\"older continuous functions.
\end{enumerate}
The common elements of the proof are the following. We decompose the process of interest using \eqref{eq:Hoeffding_centered}. The contribution of the term of index $c=1$ is basically that of a usual partial sum process based on a strictly stationary sequence. Sufficient conditions for the aforementioned results are already known. It remains to show that the contribution of the terms in \eqref{eq:Hoeffding_centered} of index $c\in \intent{2,m}$ is negligible. To do so, we use deviation inequalities for increments and for maxima of canonical $U$-statistics. These upper bounds are obtained via truncation and consist of two terms: the first one is an application of a covariance inequality for bounded canonical kernel and the second one corresponds to the unbounded part and is expressed in terms of the tail of the random variables  $h\pr{X_{\gri}}$.

As it will be explained after the statement of the theorems, our results improve existing results on the central limit theorem by giving a sharp condition and provides new results on its functional version. 

The paper is organized as follows: Section~\ref{sec:TLC} gives our result on the central limit theorem, Section~\ref{sec:TLCF_cont} gives our result on the functional central limit theorem in the space of continuous functions and Section~\ref{sec:TLC_Hold} the corresponding one for H\"older spaces. Section~\ref{sec:tools} contains the intermediate needed results, namely, the deviation inequalities and that the integrability assumption made in the statement of theorems are also satisfied by the random variables involved in the canonical $U$-statistics of Hoeffding's decomposition~\ref{eq:Hoeffding_centered}. Finally, section~\ref{sec:proofs} contains the proof of the results.

\section{Central limit theorem}\label{sec:TLC}
In order to state the main results of the paper, we need the following notations.
The generalized inverse of the tail function of a random variable $Y$ is defined as 
\begin{equation*}
    Q_Y\pr{u}=\inf\ens{t>0,\PP\pr{\abs{Y}>t}\leq u}.
\end{equation*}
 Let $\pr{X_1^*,\dots,X_{m-1}^*}$ a vector that is  independent of $\pr{X_i}_{i\in \Z}$ and that consist of i.i.d.\ random variables having the same distribution as $X_0$.
We define for $k\in\intent{1,m-1}$ the random variable 
\begin{align*}
    H_{\gri,k}&:=h\pr{X_{i_1},\dots,X_{i_k},X_1^*,\dots,X_{m-k}^*}  ,\quad \gri\in\inc^k:=\bigcup_{n\geq k}\inc^k_n\mbox{ and}\\
    H_{\gri,m}&:= 
 h\pr{X_{i_1},\dots,X_{i_m}}.
\end{align*}
These random variable arise naturally in Hoeffding's decomposition since the terms are of the form  $\E{H_{\gri,k}\mid X_{i_1},\dots,X_{i_k},X_1^*,\dots,X_j^*}$.
 
The first result of our paper is a central limit theorem for $U$-statistics based on an absolutely regular sequence. 
\begin{Theorem}\label{thm:TLC}
Let $h\colon S^m\to\R$ be a symmetric kernel, where $\pr{S,\Sca}$ is a measurable space, let $\pr{X_i}_{i\geq 1}$ be a strictly stationary sequence. Let $h_1\colon S\to \R$ be defined as 
\begin{equation}\label{eq:definition_de_h1}
    h_1\pr{x}=\int h\pr{x,y_1,\dots,y_{m-1}}d\PP_{X_0}\pr{y_1}\dots 
    d\PP_{X_0}\pr{y_{m-1}},\quad x\in S.
\end{equation}

Suppose that following conditions are satisfied:
\begin{equation}\label{eq:TLC_DMR}
    \sum_{k= 1}^\infty \int_0^{\alpha\pr{\sigma\pr{X_i,i\leq 0},\sigma\pr{X_k}}}Q^2_{ h_1\pr{X_0} }\pr{u}du<\infty,
\end{equation}
\begin{equation}\label{eq:TLC_UI}
 \mbox{the family }   \ens{H_{\gri,k}^2,k\in\intent{1,m},\gri\in\inc^k}
    \mbox{is uniformly integrable and}
\end{equation}
\begin{equation}\label{eq:condition_sur_beta}
    \sum_{k=1}^{+\infty} k\beta\pr{k}<+\infty.
\end{equation}

 Then the following convergence in distribution takes place: 
 \begin{equation*}
     \frac{\sqrt{n}}{\binom nm}\pr{U_{m,n}\pr{h}-\E{U_{m,n}\pr{h}} }
     \to m\sigma\Nca,
 \end{equation*}
 where $\Nca$ has a standard normal distribution and $\sigma^2$ 
 is given by 
 \begin{equation}\label{eq:definition_de_sigma}
    \sigma^2=\sum_{i\in\Z}\operatorname{Cov}\pr{h_1\pr{X_0},h_1\pr{X_i}}.
\end{equation}
\end{Theorem}

The assumption \ref{eq:TLC_DMR} is standard in order to guarantee  a central limit theorem for normalized partial sums of a strictly stationary sequence. The combination of \eqref{eq:TLC_UI}
and \eqref{eq:condition_sur_beta} shows that the terms of index $c\in\intent{2,m}$ in the decomposition \eqref{eq:Hoeffding_centered} are negligible.

Our result extends that of Yoshihara in the sense that the assumptions on the dependence are weaker in our result. 

A natural question is whether one can formulate a result involving only the $\alpha$-mixing coefficients and not the $\beta$-mixing coefficients, like in Theorem~1.8 in \cite{zbMATH05639709} for $U$-statistics of order $2$. To do so, one would need an appropriate covariance inequality and a generalization of the notion of $\mathcal P$-continuity given in Definition~1.4 of the aforementioned paper. This is beyond the scope of ours.

\section{Functional central limit theorem in $C[0,1]$}\label{sec:TLCF_cont}
In this section, we study the convergence in $D[0,1]$ of  the partial sum process defined by 
\begin{equation}\label{eq:definition_partial_sum_process}
W_{m,n}\pr{h,t}=\begin{cases}
\sum_{\gri\in\inc_m^{j}}h\pr{X_{\gri}},& \quad t=\frac kn\\
\mbox{linear interpolation }&\quad \mbox{on }\pr{\frac kn,\:\frac{k+1}n}, j\in\intent{0,n-1}.
\end{cases}  
\end{equation}
In the i.i.d.\ case, Miller and Sen \cite{zbMATH03378419} showed the following convergence in distribution
\begin{equation}\label{eq:WIP_iid}
   \pr{ \frac{\sqrt{n}}{\binom nm}\pr{W_{m,n}\pr{h,t}-\E{W_{m,n}\pr{h,t}} }}_{0\leq t\leq 1}
     \to \pr{m\sigma t^{m-1}B_t}_{0\leq t\leq 1}\mbox{in }C[0,1], 
\end{equation}
     where $\sigma^2=\operatorname{Var}\pr{\E{h\pr{X_1,\dots,X_m}\mid X_1}}$.
This result was extended by Denker and Keller \cite{zbMATH03822814}. It was in particular shown that if there exists some positive $\delta$ such that $\sup_{\gri\in \inc^m}\E{\abs{H_{\gri,m}}^{2+\delta}}$ is 
finite and for some $\varepsilon\in \pr{0,1/2}$,  $\beta\pr{k}^{\frac{\delta}{2+\delta}}=O\pr{k^{-2-\varepsilon}}$, then the convergence \eqref{eq:WIP_iid} holds with $\sigma$ as in \eqref{eq:definition_de_sigma}. If $\beta\pr{k}$ is of the form $k^{-\gamma}$, our condition requires that $\gamma>2$ while their condition that $\gamma\geq 2+4/\delta$. 

\begin{Theorem}\label{thm:TLCF_C}
Let $h\colon S^m\to\R$ be a symmetric kernel, where $\pr{S,\Sca}$ is a measurable space, let $\pr{X_i}_{i\geq 1}$ be a strictly stationary sequence. Let $h_1\colon S\to \R$ be defined  in 
\eqref{eq:definition_de_h1}.

Suppose that  conditions \eqref{eq:TLC_DMR}, \eqref{eq:condition_sur_beta} are satisfied, and that  
\begin{equation}\label{eq:TLCF_UI}
 \exists a>1/2\mbox{ such that }   \max_{k\in \intent{1,m}}\sup_{\gri\in\inc^ k}\E{H_{\gri,k}^2\pr{\log\pr{1+\abs{H_{\gri,k}}} }^a }<+\infty. 
\end{equation}
 Then the following convergence in distribution takes place: 
 \begin{equation}\label{eq:TLCF_cont}
    \pr{ \frac{\sqrt{n}}{\binom nm}\pr{W_{m,n}\pr{h,t}-\E{W_{m,n}\pr{h,t}} } }_{0\leq t\leq 1}
     \to \pr{m\sigma t^{m-1}B_t}_{0\leq t\leq 1}\mbox{in }C[0,1],
 \end{equation}
 where $\pr{B_t}_{0\leq t\leq 1}$ is a standard Brownian motion and $\sigma^2$ 
 is given by \eqref{eq:definition_de_sigma}.
\end{Theorem}
Here again, assumption \eqref{eq:TLC_DMR} is made in order to guarantee the convergence of the linear part in Hoeffding's decomposition and the other two to show that the remaining terms are negligible. To this aim, an application of a maximal inequality was made and gave a $\log$ factor that had to be compensated by a stronger integrability assumption. 

In \cite[Theorem~1.3]{zbMATH07902528}, a functional central limit theorem for $U$-statistics of order $2$ has been obtained under \eqref{eq:TLC_DMR}, \eqref{eq:TLCF_UI} replaced by the weaker conditions
$\max_{k\in\intent{1,m}}\sup_{\gri\in\inc^k}\E{\abs{H_{i,k}}}<+\infty$ and \eqref{eq:condition_sur_beta} replaced by the weaker condition $k^2\beta\pr{k}\to 0$. Note that the result of \cite{zbMATH07902528} addresses the case of Hilbert space-valued $U$-statistics while ours only deals with real-valued ones. On one hand, the method of \cite{zbMATH07902528} does not seem to be extendable to higher order $U$-statistics. On the other hand, our method does not seem to be extendable to vector valued kernels.

Finally, notice that our result also applies when 
$h\pr{X_{\gri}}$ does not necessarily have a moment of order $2+\delta$ for some positive $\delta$.
\section{Hölderian invariance principle} \label{sec:TLC_Hold}
 Consider the process $\pr{W_{m,n}\pr{h,t}}_{0\leq t\leq 1}$ defined by \eqref{eq:definition_partial_sum_process}. Since the process itself and the limiting process have paths that belong to the space of  $\alpha$-H\'older continuous functions, the convergence in such spaces makes sense. 

   Hölder topology gives robust estimation as pointed out in \cite{MR4127480}. It gives also the advantage to detect a change of parameter over a short period, since a test statistic based on a H\"older norm takes large values in case of change of parameter. We refer the reader to \cite{zbMATH02117810} and the references therein for more details.
   A pioneer work on invariance principle in H\"older spaces for sequences of random variables has been done by Ra\v{c}kauskas and Suquet \cite{zbMATH02118866,zbMATH02083467}, where a necessary and sufficient condition for the convergence of $\pr{\sqrt n\pr{W_{1,n}\pr{t}-\E{W_{1,n}\pr{t}}}}_{0\leq t\leq 1}$ in H\"older having a modulus of regularity of the form $t^\alpha L\pr{1/t}$ where $\alpha\in (0,1/2)$ and under some assumptions for $L$ has been given. When $L$ is constant, the necessary and sufficient condition reads as 
   \[
   \lim_{t\to +\infty}t^{p\pr{\alpha}}
   \PP\pr{\abs{h\pr{X_1}}>t}=0, \quad p\pr{\alpha}=\frac 1{\frac 12-\alpha}.
   \]
   Several works investigated the convergence of such partial sum processes for stationary weakly dependent sequences: mixing sequences \cite{zbMATH06705457}, $\alpha$-dependent sequences \cite{zbMATH06673296}, projective condition in the spirit of Hannan \cite{zbMATH06518045}, Maxwell and Woodroofe \cite{zbMATH06973953}.

 We denote by $\Hca_\alpha\pr{[0,1]}$ the space of Hölder continuous 
functions on $[0,1]$, that is, the set of functions $x\colon [0,1]\to 
\R$ such that 
\begin{equation*}
\norm{x}_\alpha:=\abs{x\pr{0}}+ 
\sup_{s,t\in[0,1],s<t}\frac{\abs{x\pr{s}-x\pr{t}}}{\pr{s-t }^\alpha}<\infty.
\end{equation*}
In \cite{MR4880702}, Theorem~2.3, it was shown that when the sequence $\pr{X_i}_{i\in \Z}$ is i.i.d., 
a sufficient condition for the convergence of 
$\pr{ \frac{\sqrt{n}}{\binom nm}\pr{W_{m,n}\pr{h,t}-\E{W_{m,n}\pr{h,t}}} }_{0\leq t\leq 1}$ in $\Hca_\alpha$ is
\begin{equation}\label{eq:cond_suffisante_PI_holderien_Ustats}
 \lim_{t\to\infty}t^{p\pr{\alpha}}\PP\pr{\abs{h\pr{X_1,\dots,X_m} }>t   
}=0.
\end{equation}
As far as we know, the dependent case has not been addressed yet. Under \eqref{eq:condition_sur_beta},  a reinforced version of \eqref{eq:TLC_DMR} and a condition
in the spirit of \eqref{eq:cond_suffisante_PI_holderien_Ustats}, we can show the wanted convergence. 
 Let 
\begin{equation*}
    \alpha^{-1}\pr{u}=\inf\ens{k\in\N^*, \alpha\pr{k}\leq u}.
\end{equation*}
\begin{Theorem}\label{thm:TLC_Holder}
Let $\alpha\in\pr{0,1/2}$ and define $p\pr{\alpha}:=\pr{1/2-\alpha}^{-1}$. 
Suppose that \eqref{eq:condition_sur_beta} is satisfied as well as the the following conditions:
\begin{equation}\label{eq:DMR_condition_Holder}
    \lim_{s\to +\infty} s^{p\pr{\alpha}-1}
    \int_0^1 Q_{h_1\pr{X_0}}\pr{u}\mathbf{1}\ens{\alpha^{-1}\pr{u}Q_{h_1\pr{X_0}}\pr{u} >s}du=0
\end{equation}
\begin{equation}\label{eq:tail_condition_Holder}
    \lim_{t\to +\infty }t^{p\pr{\alpha}}\max_{k\in \intent{1,m} }\sup_{\gri\in \inc^k}
    \PP\pr{\abs{ H_{\gri,k}}>t }=0.
\end{equation}
Then the following convergence in distribution takes place: 
 \begin{equation*}
 \pr{ \frac{\sqrt{n}}{\binom nm}\pr{W_{m,n}\pr{h,t}-\E{W_{m,n}\pr{h,t}}} }_{0\leq t\leq 1}
     \to \pr{m\sigma t^{m-1}B_t}_{0\leq t\leq 1}\mbox{ in }\Hca_\alpha \pr{[0,1]},
 \end{equation*}
 where $\pr{B_t}_{0\leq t\leq 1}$ is a standard Brownian motion and $\sigma^2$ 
 is given by \eqref{eq:definition_de_sigma}.
\end{Theorem}
Condition \eqref{eq:DMR_condition_Holder} is required for the convergence of the linear part, and the combination of \eqref{eq:tail_condition_Holder} with \eqref{eq:condition_sur_beta} in order to show the negligibility of the remaining terms.

\section{Tools for the proofs}\label{sec:tools}
\subsection{Second moment inequalities}

In order to prove Proposition~\ref{prop:neg_terme_biais}, we need the following lemma which follows, as pointed out in \cite{MR1624866}, from the characterization of $\beta$-mixing coefficients given page 193 in \cite{MR137141} and an induction argument given in Lemma~2 
in \cite{MR777842}. See also Lemma~1 in \cite{MR418179}.
\begin{Lemma}\label{lem:comparaison_esperances}
Let $\pr{X_j}_{j\geq 1}$ be a strictly stationary sequence of random variables taking values in a measurable space $\pr{S,\Sca}$. Let $f\colon S^m\to \R$ be a measurable function, where $S^m$ is endowed with the product $\sigma$-algebra. 
Let $\pr{m\pr{u,v}}_{\substack{1\leq u\leq \ell\\ 1\leq v\leq r_u}}$ be integers such that 
\begin{equation*}
    m\pr{1,1}<\dots<m\pr{1,r_1}<m\pr{2,1}<\dots< m\pr{2,r_2}<\dots< m\pr{\ell,1}<\dots<m\pr{\ell,r_\ell}.
\end{equation*}
Let $r=\sum_{u=1}^\ell r_u$. Let $\pr{\xi_j}_{j=1}^r$ be a sequence of random variables such that for each $j$, $\xi_j$ has the same distribution as $X_1$ and 
\begin{multline*}
    \mathcal{L}\pr{\pr{\xi_{m\pr{1,1}},\dots,\xi_{m\pr{1,r_1}},\xi_{m\pr{2,1}},
\dots,\xi_{m\pr{2,r_2}},\dots,\xi_{m\pr{\ell,1}},\dots,\xi_{m\pr{\ell,r_\ell}}  }}\\
=\mathcal{L}\pr{X_{m\pr{1,1}},\dots,X_{m\pr{1,r_1}} }
\otimes \dots\otimes \mathcal{L}\pr{X_{m\pr{\ell,1}},\dots,X_{m\pr{\ell,r_\ell}} }.
\end{multline*}
Suppose that there exists a constant $R$ such that for each $x_1,\dots,x_m\in S$,  $\abs{f\pr{x_1,\dots,x_m}}\leq R$.
Then the following inequality takes place
\begin{multline}\label{eq:diff_esperances_f_bornee}
    \abs{\E{f\pr{X_{m\pr{1,1}},\dots,X_{m\pr{\ell,r_\ell}} }}-
    \E{f\pr{\xi_{m\pr{1,1}},\dots,\xi_{m\pr{\ell,r_\ell}} }}}
\\    \leq 2\sum_{i=1}^{\ell-1}
    \beta\pr{m\pr{i+1,1}-m\pr{i,r_i}}R.
\end{multline}
\end{Lemma}

 As pointed out before, a fundamental tool to control the contribution of the terms in Hoeffding's decomposition is an inequality that bounds the moment of 
 order two of increments of a $U$-statistic with canonical kernels. 
 This allows the derivation of a maximal inequality that plays a crucial role 
 in functional limit theorems. The obtained bound is expressed 
 in terms of the $\beta$-mixing coefficients and a truncation level $R$. Let us start by the bounded case. 
\begin{Proposition}\label{prop:covariance_inequality}
 For each $m\geq 2$, there exists a constant $C_m$ such that if  $\pr{S,\Sca}$ is a measurable space, $\pr{X_i}_{i\in\Z}$ is a strictly stationary sequence taking   $h\colon S^m\to \R$ is a canonical kernel such that $\sup_{x_1,\dots,x_m\in S}\abs{h\pr{x_1,\dots,x_m}}\leq R$ and $N\geq m$, then 
 \begin{align}\label{eq:inegalite_accroissements_h_canonique_h_bornee}
     \E{\pr{U_{m,N+a}\pr{h}-U_{m,a}\pr{h}}^2}&
     \leq C_{m} \pr{N+a}^{ m-1 } N\pr{1+\sum_{k=1}^{N+a}k^{m-1} 
      \beta\pr{k} }R^2,\\
     \E{\max_{m\leq n\leq N}\pr{U_{m,n}\pr{h}   }^2    }
     &\leq C_mN^{m }\pr{\log  N}^2\pr{1+\sum_{k=1}^Nk^{m-1}
      \beta\pr{k} }R^2.\label{eq:inegalite_max_h_canonique_h_bornee}
 \end{align}
\end{Proposition}

\begin{proof}[Proof of Proposition~\ref{prop:covariance_inequality}]
Let us start by showing \eqref{eq:inegalite_accroissements_h_canonique_h_bornee}. The proof follows the ideas of that of Lemma~3 in \cite{MR1624866}, with the main difference that  we have to control the second moment of a difference of a $U$-statistic. We will rest mainly on Arcones' proof and explain what has to be changed.

One has 
\[
 \E{\pr{U_{m,N+a}\pr{h}-U_{m,a}\pr{h}}^2}=
 \sum_{\gri\in\inc^m_{N+a}: i_m\geq a}
 \sum_{\gr{i'}\in\inc^m_{N+a}: i'_m\geq a}
 \E{h\pr{X_{\gri}}h\pr{X_{\gr{i'}}}}.
\]
which can be rewritten  in terms of permutations as 
\begin{multline}\label{eq:etape_1_inegalite_cov}
 \E{\pr{U_{m,N+a}\pr{h}-U_{m,a}\pr{h}}^2}\\
 \leq \sum_{\sigma\in\Sca\pr{2m}}\sum_{\substack{1\leq i_1\leq i_2\leq \dots \leq i_{2m}\leq N+a\\ i_{\sigma\pr{1}}<\dots <i_{\sigma\pr{m}}, i_{\sigma\pr{m}}>a\\
 i_{\sigma\pr{m+1}}<\dots <i_{\sigma\pr{2m}}, i_{\sigma\pr{2m}}>a}}
 \E{h\pr{X_{i_{\sigma\pr{1}}},\dots, X_{i_{\sigma\pr{m}}}  }  h\pr{X_{i_{\sigma\pr{m+1}}},\dots, X_{i_{\sigma\pr{2m}}}}},
 \end{multline}
 where $\Sca\pr{2m}$ denotes the set of permutations from $\intent{1,2m}$ to itself. Let $\sigma\in\Sca\pr{2m}$ be fixed.
Let $I_{\sigma}$ be the set of $\gri:=\pr{i_u}_{u=1}^{2m}$ such that $1\leq i_1\leq i_2\leq \dots \leq i_{2m}\leq N+a$, $i_{\sigma\pr{1}}<\dots <i_{\sigma\pr{m}}, i_{\sigma\pr{m}}>a$, $i_{\sigma\pr{m+1}}<\dots <i_{\sigma\pr{2m}}$ and  $i_{\sigma\pr{2m}}>a$. Define for $\gri:=\pr{i_u}_{u=1}^{2m}\in I$ the integers
 \begin{align*}
     j_1\pr{\gri}&=i_2-i_1,\\
     j_\ell\pr{\gri}&=\min\ens{i_{2\ell-1}-i_{2\ell-2},i_{2\ell}-i_{2\ell -1}},\quad \ell\in\intent{2,m-1}\\
     j_m\pr{\gri}&=i_{2m}-i_{2m-1}.
 \end{align*}

     In view of \eqref{eq:etape_1_inegalite_cov}, 
     it suffices to show that there exists a constant $C_{m} $ such that for each $\sigma\in\Sca\pr{2m}$ and  each 
     $\ell_0\in\intent{1,m}$,
     \begin{multline}\label{eq:etape_2_inegalite_cov}
       \sum_{\substack{\gri:=\pr{i_u}_{u=1}^{2m}\in I_{\sigma}  \\ 
       j_{\ell_0}\pr{\gri}\geq \max_{1\leq \ell\leq m}
       j_\ell\pr{\gri}
       } }  \E{h\pr{X_{i_{\sigma\pr{1}}},\dots, X_{i_{\sigma\pr{m}}}  }  h\pr{X_{i_{\sigma\pr{m+1}}},\dots, X_{i_{\sigma\pr{2m}}}}}\\
          \leq C_{m}\pr{N+a}^{m-1}N\pr{1+ \sum_{k=1}^{N+a}k^{m-1} \beta\pr{k}R^2}.
     \end{multline}
     We first consider the case where the indexes are not repeated.
We will apply Lemma~\ref{lem:comparaison_esperances} to the function $f\colon S^{2m}\to \R$ defined by 
\begin{equation*}
    f\pr{x_1,\dots,x_{2m}}=h\pr{x_1,\dots,x_m}
    h\pr{x_{m+1},\dots,x_{2m}}.
\end{equation*}
Note that if there is a block which consists of only one random variable, that is, if one of the $r_u$ is equal to one and the corresponding index is not repeated, then $\E{f\pr{\xi_{m\pr{1,1}},\dots,\xi_{m\pr{\ell,r_\ell}} }}=0$. 

Suppose that $\ell_0=1$. Let $\gri\in I_\sigma$. We use 
Lemma~\ref{lem:comparaison_esperances} with the blocks $\ens{i_1}$, $\ens{i_2,\dots,i_m}$. 
 Let $k:=i_2-i_1\in\intent{1,N+a}$. Then $\pr{i_1,i_2}$ can take at most $N+a$ different values. 
 Let $\ell\in\intent{2,m-1}$. Suppose that  $i_{2\ell-1}-i_{2\ell-2}<i_{2\ell}-i_{2\ell -1}$. 
 Since $j_1\pr{\gri}\geq j_\ell\pr{\gri}=i_{2\ell-1}-i_{2\ell-2}$, we have that 
 $i_{2\ell-1}-i_{2\ell-2}\leq k$ hence $i_{2\ell-1}$ can take at most $k$ values and $i_{2\ell}$ at most $N+a$. If 
 $i_{2\ell-1}-i_{2\ell-2}\geq i_{2\ell}-i_{2\ell -1}$. 
 then $i_{2\ell}$ can take at most $k$ values and 
 $i_{2\ell-1}$ at most $N+a$. 
 Notice also that $i_{2m}-i_{2m-1}\leq i_2-i_1$ hence there are at most $k$ possibilities for $i_{2m-1}$ and $N$ for $i_{2m-1}$ since $i_{2m}$ corresponds to the largest element of $\ens{i_{\sigma\pr{m} }, i_{\sigma\pr{2m}}}$, which is in any case between $a$ and $N+a$. The number of summands does not exceed $\pr{N+a}^{m-1}Nk^{m-1}$  which are all, by Lemma~\ref{lem:comparaison_esperances}, smaller than 
 $C_m\beta\pr{k}R^2$.

Suppose that $\ell_0\in\intent{2,m-1}$. Let $\gri\in I_{\sigma}$. This time, we apply Lemma~\ref{lem:comparaison_esperances} with the blocks 
$\ens{i_1,\dots,i_{2\ell-2}}$, $\ens{i_{2\ell-1}}$ and 
$\ens{i_{2\ell},\dots,i_{2m}}$. Let $k:=j_{\ell_0}\pr{\gri}$.
As before, there are $\pr{N+a}k$ possibilities for $\pr{i_1,i_2}$, as well as for the terms $\pr{i_{2\ell-1},i_{2\ell}}$ for  $\ell\in\intent{2,m-1}\setminus\ens{\ell_0}$. 
For the terms $i_{2\ell_0-1}$ and $i_{2\ell_0}$, we have that 
if $i_{2\ell_0-1}-i_{2\ell_0-2}>i_{2\ell_0}-i_{2\ell_0-1}$, 
there are $k$ possibilities for $i_{2\ell_0}$ and $N+a$ for $i_{2\ell_0-1}$ and the other way around if the previous inequality is reversed.
Like before, there are at most $k N$ possibilities for $\pr{i_{2m-1},i_{2m}}$ hence \eqref{eq:etape_2_inegalite_cov} 
also holds in this case. 

Finally, suppose that $\ell_0=m$ and let $\gri\in I_{\sigma}$.
We apply Lemma~\ref{lem:comparaison_esperances} with the blocks $\ens{i_1,\dots,i_{2m-1}}$ and $\ens{i_{2m}}$.
Let $k:= i_{2m}-i_{2m-1}$.
For $i_1,\dots,i_{2m-2}$, we have $\pr{N+a}^{m-2}k^{m-2}$ and as before, the number of possibilities for $\pr{i_{2m-1},i_{2m}}$ is at most $k N$.

The case where there are repeated indexes can be treated in a similar way. Let us illustrate it in two cases. When all the indexes are repeated, we have at most $\pr{N+a}N^{m-1}$ terms
that are bounded by $\sup_{\gri\in\inc^m}\norm{h\pr{X_{\gri}}}_2^2$, hence the contribution does not exceed $\pr{N+a}N^{m-1}R^2$. Let us now deal with the case where only the first  indices are now repeated.  We thus have to control terms of the form 
\[\sum_{1\leq i_1\neq j_1<j_2<\dots<j_m\leq N+a, j_m\geq a}
\E{h\pr{X_{i_1},X_{j_2},\dots,X_{j_m}}
h\pr{X_{j_1},X_{j_2},\dots,X_{j_m}}}.\]
This time, we cannot use Lemma~\ref{lem:comparaison_esperances} with blocks of the form $\ens{i_1,j_1,\dots,j_k}$ and 
$\ens{j_{k+1}}$ because the resulting random variable with the independent blocks would not be centered. Suppose without loss of generality that $i_1<j_1$ and let 
$k:=\max\ens{j_1-i_1,j_2-j_1}$. If $j_1-i_1<j_2-j_1$ we use Lemma~\ref{lem:comparaison_esperances} with the blocks 
$\ens{i_1,j_1}$ and $\ens{j_2,\dots,j_m}$ and if $j_1-i_1\geq j_2-j_1$, with the blocks $\ens{i_1}$ and $\ens{j_1,\dots,j_m}$. 
The number of involved terms is $N$ for $j_m$, $\pr{N+a}^{m-1}$ 
for the indexes $j_2,\dots,j_{m-1}$ and for $\pr{i_1,j_1}$, we have at most $k\pr{N+a}$ possibilities. The remaining cases, namely, where some indexes are repeated and others not can be treated in the same way.
This ends the proof of \eqref{eq:inegalite_accroissements_h_canonique}.

Let us now prove \eqref{eq:inegalite_max_h_canonique_h_bornee}. 
We will use the main result of \cite{MR268938}, which reads as follows. Let $\pr{Y_i}_{i\geq 1}$ be a sequence of random variables having a finite moment of order $2$. Denote $S_{a,k}:=\sum_{i=a+1}^{a+k}Y_i$, $M_{a,n}:=\max_{1\leq k\leq n}\abs{S_{a,k}}$. Suppose that there exists a double index sequence 
$\pr{c_{a,k}}_{a,k\geq 1}$ such that 
\begin{equation}\label{eq:condition_double_indexed}
    c_{a,k}+c_{a+k,\ell}\leq c_{a,k+\ell}
\end{equation}
and for each $a\geq a_0$, $n\geq 1$, $\E{S_{a,n}^2}\leq c_{a,n}$. Then 
for each $a\geq a_0$, $n\geq 1$, $\E{M_{a,n}^2}\leq \pr{\log_2\pr{2n}}^2c_{a,n}$ (in \cite{MR268938}, the sequence is assumed to be a function of the distribution function of $Y_{a+1},\dots,Y_{a+n}$ but actually, only the fact that $\pr{c_{a,k}}_{a,k\geq 1}$ satisfies the inequalities  \eqref{eq:condition_double_indexed} and  $\E{S_{a,n}^2}\leq c_{a,n}$ matters). Define 
\begin{equation*}
    c_{a,n}:=C_{m} \pr{n+a}^{ m-1 } n\pr{1+\sum_{i=1}^{n+a}i^{m-1}\beta\pr{i}}R^2, 
\end{equation*}
where $C_m$ is as in \eqref{eq:inegalite_accroissements_h_canonique}
and $Y_i=\sum_{1\leq j_1<\dots<j_{m-1}<i}h\pr{X_{j_1},\dots,X_{j_{m-1}},X_i  }$. We have already seen that $\E{S_{a,n}^2}\leq c_{a,n}$ holds. It remains to check that $\pr{c_{a,k}}_{a,k\geq 1}$ satisfies \eqref{eq:condition_double_indexed}. Clearly, this condition is scale invariant, hence it suffices to show it is satisfied for 
\begin{equation*}
    c'_{a,n}:=  \pr{n+a}^{ m-1 } n\pr{1+\sum_{i=1}^{n+a}i^{m-1}
      \beta\pr{i} }  ,
\end{equation*}
One has 
\begin{align*}
    c'_{a,k}+c'_{a+k,\ell}&= \pr{k+a}^{ m-1 } k\pr{1+\sum_{i=1}^{k+a}i^{m-1}
      \beta\pr{i}}
      + \pr{\ell+a+k}^{ m-1 } \ell\pr{1+\sum_{i=1}^{\ell+a+k}i^{m-1}
     \beta\pr{i} } \\
      &\leq \pr{1+\sum_{i=1}^{\ell+a+k}i^{m-1}
      \beta\pr{i} }\pr{ \pr{k+a}^{ m-1 } k+ \pr{\ell+a+k}^{ m-1 }\ell } \\
      &\leq \pr{1+\sum_{i=1}^{\ell+a+k}i^{m-1}
      \beta\pr{i} }\pr{ \pr{k+a
      +\ell}^{ m-1 } k+ \pr{\ell+a+k}^{ m-1 }\ell }\\
      &=c'_{a,k+\ell}.
\end{align*}
This ends the proof of Proposition~\ref{prop:covariance_inequality}.
\end{proof}

Using truncation arguments, we can formulate the following tail inequality for $U$-statistic based on a canonical kernel.
\begin{Proposition}\label{prop:tail_inequality}
For each $m\geq 2$, there exists a constant $\kappa_m$ such that if  $\pr{S,\Sca}$ is a measurable space, $\pr{X_i}_{i\in\Z}$ is a strictly stationary sequence taking values in $S$,   $h\colon S^m\to \R$ is a canonical kernel, $a\geq 0$ and $N\geq m$, the following inequalities hold for each positive $R,t$: 
 \begin{multline}\label{eq:inegalite_accroissements_h_canonique}
 \PP\pr{\abs{U_{m,N+a}\pr{h}-U_{m,a}\pr{h}}> t}
  \\ 
     \leq R^2\frac{\kappa_{m}}{t^2} N\pr{N+a}^{2m-3}  \pr{1+\sum_{k=1}^{N+a }k
      \beta\pr{k}}+\frac{\kappa_m}{t} \frac{\pr{N+a}^{m-1}}{N^{m-3/2}} 
      \sup_{\gri\in\inc^m}\E{ 
      \abs{h\pr{X_{\gri}}}\mathbf{1}_{\abs{h\pr{X_{\gri}}}>R}
      },
 \end{multline}
 \begin{multline}\label{eq:inegalite_max_h_canonique}
     \PP\pr{\max_{m\leq n\leq N}\abs{U_{m,n}\pr{h}   }>N^{m-1/2}t    } \\
     \leq     \frac{\kappa_{m}}{t^2} \pr{\log N}^2\frac{R^2}N \pr{1+\sum_{k=1}^{N }k
      \beta\pr{k}}+\frac{\kappa_m}{t} \sqrt{N} 
      \sup_{\gri\in\inc^m}\E{ 
      \abs{h\pr{X_{\gri}}}\mathbf{1}_{\abs{h\pr{X_{\gri}}}>R}
      }.
 \end{multline}
\end{Proposition}
\begin{proof}
Define the kernels
\begin{align*}
  h_{\leq }&=\pr{x_1,\dots,x_m}:=h\pr{x_1,\dots,x_m}
\mathbf{1}_{\abs{h\pr{x_1,\dots,x_m}}\leq R}  \\
h_{>}&=\pr{x_1,\dots,x_m}:=h\pr{x_1,\dots,x_m}
\mathbf{1}_{\abs{h\pr{x_1,\dots,x_m}}> R}.
\end{align*}
Obviously $h=h_{\leq}+h_{>}$, but the kernels $h_{\leq }$ and $h_{>}$ 
are not canonical. Therefore, we define 
\begin{align}\label{eq:noyaux_deg}
    h_{\leq}^{\operatorname{can}}\pr{x_1,\dots,x_m}
&=\pr{\delta_{x_1}-\PP_{X_1}}\dots \pr{\delta_{x_m}-\PP_{X_1}}\pr{h_{\leq}}\\
   h_{>}^{\operatorname{can}}\pr{x_1,\dots,x_m}
&=\pr{\delta_{x_1}-\PP_{X_1}}\dots \pr{\delta_{x_m}-\PP_{X_1}}\pr{h_{>}}.
\end{align}
Since $h=h_{\leq}^{\operatorname{can}}+h_{>}^{\operatorname{can}}$, 
we derive that 
\begin{multline*}
    \PP\pr{\abs{U_{m,N+a}\pr{h}-U_{m,a}\pr{h}}> t}\\ \leq 
\PP\pr{\abs{U_{m,N+a}\pr{h_{\leq}^{\operatorname{can}}}-
U_{m,a}\pr{h_{\leq}^{\operatorname{can}}}}> 
\frac t2}+\PP\pr{\abs{U_{m,N+a}\pr{h_{>}^{\operatorname{can}}}-U_{m,a}\pr{h_{>}^{\operatorname{can}}}}> 
\frac t2}.
\end{multline*}
Using Tchebychev's inequality for the first term, Markov's and the triangle inequality for the second one, we derive that 
\begin{multline}\label{eq:proof_tail_bound}
\PP\pr{\abs{U_{m,N+a}\pr{h}-U_{m,a}\pr{h}}>N^{m-1/2}t}\leq
    \frac{4}{t^2 }\E{\pr{
    U_{m,N+a}\pr{h_{\leq}^{\operatorname{can}}}-
        U_{m,a}\pr{h_{\leq}^{\operatorname{can}}}}^2}\\
    +\frac{2}{t }\sum_{\gri\in\inc^m_{N+a}, i_m\geq a}
    \E{\abs{h_{>}^{\operatorname{can}}\pr{X_{\gri}}  }}.
\end{multline}
Notice that $h_{\leq}^{\operatorname{can}}$ is bounded by $2^mR$ and $\E{\abs{h_{>}^{\operatorname{can}}\pr{X_{\gri}} } }$ can be bounded, by symmetry,
by terms of the form 
$\E{\abs{h_{>} \pr{X_{i_1},\dots,X_{i_k},Y_1,\dots,Y_{m-k}}  }}$, 
where $Y_1,\dots,Y_{m-k} $ are independent random variables. 
For such terms, the following inequality holds 
\begin{equation}\label{eq:bound_expectation}
  \E{\abs{h_{>} \pr{X_{i_1},\dots,X_{i_k},Y_1,\dots,Y_{m-k}  }}}\leq   \sup_{\gri\in\inc^m}\E{\abs{h_{>}\pr{X_{\gri}}}}.
\end{equation} 
 Indeed,  using recursively Berbee's coupling lemma (see \cite{MR547109}), we can find independent random variables $X_{k+1}^*,\dots,X_m^*$ having the same distribution as 
$X_1$ and such that $\PP\pr{X_{k+u}^*\neq X_{i_{k}+u\ell}}\leq \beta\pr{\ell}$ for each $u\in\intent{1,m-k}$. Let 
$E_\ell$ be the event $\bigcap_{u=1}^{m-k}\ens{X_{k+u}^*= X_{i_{k}+u\ell}} $. For each $\ell$, 
\begin{align*}
\sup_{\gri\in\inc^m}\E{\abs{h_{>}\pr{X_{\gri}}}}&\geq 
\E{\abs{h_{>}^{\operatorname{can}}\pr{X_{i_1},\dots,X_{i_k}, 
X_{i_k+\ell},\dots,X_{i_k +\pr{m-k}\ell} } }  }\\
&=\E{\abs{h_{>}^{\operatorname{can}}\pr{X_{i_1},\dots,X_{i_k}, 
X_{i_k+\ell},\dots,X_{i_k +\pr{m-k}\ell} } }\mathbf{1}_{E_\ell}  }\\
&+\E{\abs{h_{>}^{\operatorname{can}}\pr{X_{i_1},\dots,X_{i_k}, 
X_{i_k+\ell},\dots,X_{i_k +\pr{m-k}\ell} } }\mathbf{1}_{E_\ell^c}}\\
 &=\E{\abs{h_{>}^{\operatorname{can}}\pr{X_{i_1},\dots,X_{i_k}, 
X_{k+1}^*,\dots,X_{m}^* } }\mathbf{1}_{E_\ell}  } \\
&+\E{\abs{h_{>}^{\operatorname{can}}\pr{X_{i_1},\dots,X_{i_k}, 
X_{i_k+\ell},\dots,X_{i_k +\pr{m-k}\ell} } }\mathbf{1}_{E_\ell^c}}\\
 &\geq \E{\abs{h_{>}^{\operatorname{can}}\pr{X_{i_1},\dots,X_{i_k}, 
X_{k+1}^*,\dots,X_{m}^* } }   }\\
&-
\E{\abs{h_{>}^{\operatorname{can}}\pr{X_{i_1},\dots,X_{i_k}, 
X_{k+1}^*,\dots,X_{m}^* } }\mathbf{1}_{E_\ell^c}  }
 \end{align*}
We thus derived that 
\begin{multline}
    \E{\abs{h_{>}^{\operatorname{can}}\pr{X_{i_1},\dots,X_{i_k}, 
X_{k+1}^*,\dots,X_{m}^* } }   }
\\
\leq \E{\abs{h_{>}^{\operatorname{can}}\pr{X_{i_1},\dots,X_{i_k}, 
X_{k+1}^*,\dots,X_{m}^* } }\mathbf{1}_{E_\ell^c}  }+
\sup_{\gri\in\inc^m}\E{\abs{h_{>}\pr{X_{\gri}}}}
\end{multline}
and we get \eqref{eq:bound_expectation} by letting $\ell$ going to 
infinity.
 
The first term of the right hand side of \eqref{eq:proof_tail_bound} 
is estimated by Proposition~\ref{prop:covariance_inequality} and for 
the second term, we use the triangle inequality and 
\eqref{eq:bound_expectation}
to derive that 
\begin{multline}\label{eq:proof_tail_bound2}
    \PP\pr{\abs{U_{m,N+a}\pr{h}-U_{m,a}\pr{h}}> t}\leq 
    \frac{2^{2m+2}C_m}{t^2 }R^2 \pr{N+a}^{m-1}N\pr{1+\sum_{k=1}^{N+a}k^{m-1}\beta\pr{k}}\\ 
   + \frac{2^{m+1}}{t }\pr{N+a}^{m-1}N 
\sup_{\gri\in\inc^m}\E{\abs{h\pr{X_{\gri}}} \mathbf{1}_{\abs{h\pr{X_{\gri}}}>R}} .
\end{multline}
Bounding $k^{m-1}$ by $\pr{N+a}^{m-2}k$ and letting $\kappa_m=2^{2m+2}C_m+2^{m+1}$
end  the proof of \eqref{eq:inegalite_accroissements_h_canonique}.

To complete the proof of Proposition~\ref{prop:tail_inequality}, it remains to check \eqref{eq:inegalite_max_h_canonique}. This follows the same lines: we decompose $h$ as $h_{\leq}^{\operatorname{can}}+h_{>}^{\operatorname{can}}$
and use Tchebychev's and Markov's inequality combined with the elementary bound $\max_{m\leq n\leq N}\abs{\sum_{\gri\in\inc^m_n}A_{\gri}}\leq \sum_{\gri\in\inc^m_N}\abs{A_{\gri}}$ to get 
\begin{multline*}
    \PP\pr{\max_{m\leq n\leq N}\abs{U_{m,n}\pr{h} }>N^{m-1/2}t}\\ \leq
\PP\pr{\max_{m\leq n\leq N}\abs{U_{m,n}\pr{h_{\leq}^{\operatorname{can}}} }>N^{m-1/2}
\frac t2}+\PP\pr{\abs{\max_{m\leq n\leq N}\abs{U_{m,n}\pr{h_{>}^{\operatorname{can}}} }}>N^{m-1/2}
\frac t2}\\
\leq \frac{4}{t^2N^{2m-1}}\E{ \max_{m\leq n\leq N}\pr{U_{m,n}\pr{h_{\leq}}}^2 }
+\frac{2}{tN^{m-1/2}}\sum_{\gri\in\inc^m_N}\E{\abs{h_{>}^{\operatorname{can}}\pr{X_{\gri}}  }}.
\end{multline*}
Then applying \eqref{eq:inegalite_max_h_canonique_h_bornee} for the first term and treating the second one as before gives  
\begin{multline*}
    \PP\pr{\max_{m\leq n\leq N}\abs{U_{m,n}\pr{h} }>N^{m-1/2}t}\\  
\leq \frac{4}{t^2N^{2m-1}}\pr{C_mN^m \pr{\log_2 N}^2\pr{1+\sum_{k=1}^Nk^{m-1}\beta\pr{k} } }
+\frac{2}{tN^{m-1/2}}2^{m+1} N^m\sum_{\gri\in\inc^m_N}\E{\abs{h_{>}^{\operatorname{can}}\pr{X_{\gri}}  }}.
\end{multline*}
Bounding $k^{m-1}$ by $N^{m-2}k$ end the proof of Proposition~\ref{prop:tail_inequality}.
\end{proof}
\subsection{Treatment of biais terms}

Since in general, the random variables  $h_{c,m}\pr{X_{\gri}}$, $\gri\in\inc^c$, involved in Hoeffding's decomposition are not centered, 
their expectation in the centered $U$-statistics does not vanished. Nevertheless, under appropriate conditions on integrability of the random variables $h\pr{X_{\gri}},\gri\in\inc^m$ and on the $\beta$-mixing coefficients, the contribution of the expectations of the random variables $h_{c,m}\pr{X_{\gri}}$, $\gri\in\inc^c$, is negligible. 
\begin{Proposition}\label{prop:neg_terme_biais}
Let $h\colon S^m\to\R$ be a canonical kernel. Suppose that $\sum_{k=1}^{+\infty} k\beta\pr{k}$ is finite. 
Suppose that $\ens{h\pr{X_{\gri}}^2,\gri\in\inc^m}$ is uniformly integrable.
% mettre la bonne condition 
Then 
\begin{equation}\label{eq:convergence_esperances_TLC}
    \lim_{n\to +\infty}\frac 1{n^{m-1/2  }}\sum_{\gri\in\inc^m_n}
    \abs{\E{h\pr{X_{\gri}}}}=0.
\end{equation}
Let $\alpha\in\pr{0,1/2}$. Suppose that 
\begin{equation*}
    \sup_{t>0}\sup_{\gri\in\inc^m}t^{p\pr{\alpha}}
    \PP\pr{\abs{h\pr{X_{\gri}}}>t }<+\infty,
\end{equation*}
where $p\pr{\alpha}=\pr{1/2-\alpha}^{-1}$. Then
\begin{equation}\label{eq:convergence_esperances_TLCF_Holder}
    \lim_{n\to +\infty}\frac 1{n^{m-1/2-\alpha }}\sum_{\gri\in\inc^m_n}
    \abs{\E{h\pr{X_{\gri}}}}=0.
\end{equation}
\end{Proposition}
\begin{proof}[Proof of Proposition~\ref{prop:neg_terme_biais}]
Let $\ell\in\intent{1,m}$ and $\gri\in\inc^m_n$ be fixed. Let also $h_{\leq}^{\operatorname{can}}$ and $h_{>}^{\operatorname{can}}$ be defined by \eqref{eq:noyaux_deg}
We apply Lemma~\ref{lem:comparaison_esperances} with the blocks $\ens{i_1,\dots,i_{\ell-1}}$, $\ens{i_\ell}$ and $\ens{i_{\ell+1},\dots,i_m}$ and the function $h_{\leq}^{\operatorname{can}}$. We get that 
\[
\frac 1{n^{m-1/2  }}\sum_{\gri\in\inc^m_n}
    \abs{\E{h\pr{X_{\gri}}}}\leq n^{-1/2}\sum_{k=1}^nk\beta\pr{k}R +n^{1/2}\sup_{\gri\in\inc^m_n}
    \E{\abs{h\pr{X_{\gri}}}\mathbf{1}_{\abs{h\pr{X_{\gri}}}>R} }.
\]
For a fixed $\eps$, we choose $R=\eps n^{1/2}/\sum_{k=1}^{+\infty}k\beta\pr{k} $ and we use 
the elementary inequality $\E{Y\mathbf{1}_{Y>R}}\leq R^{-1}
\E{Y^2\mathbf{1}_{Y>R}}$
so that the previous bound becomes 
\[
\frac 1{n^{m-1/2  }}\sum_{\gri\in\inc^m_n}
    \abs{\E{h\pr{X_{\gri}}}}\leq \eps +\frac1{\eps}\sum_{k=1}^{+\infty}k\beta\pr{k}
     \sup_{\gri\in\inc^m_n}
    \E{h\pr{X_{\gri}}^2\mathbf{1}_{\abs{h\pr{X_{\gri}}}>
    \eps n^{1/2 }/\sum_{k=1}^{+\infty}k\beta\pr{k}} }.
\]
    Uniform integrability of $\ens{h\pr{X_{\gri}}^2,\gri\in\inc^m}$ gives that for each positive $\eps$, 
    \[
    \limsup_{n\to +\infty}\frac 1{n^{m-1/2  }}\sum_{\gri\in\inc^m_n}
    \abs{\E{h\pr{X_{\gri}}}}\leq \eps
    \]
    from which \eqref{eq:convergence_esperances_TLC}. 
    The proof of \eqref{eq:convergence_esperances_TLCF_Holder} follows exactly the same lines, replacing the normalization $n^{m-1/2}$ by $n^{m-1/2-\alpha}$ and the bound  
    $\E{Y\mathbf{1}_{Y>R}}\leq R^{-1}
\E{Y^2\mathbf{1}_{Y>R}}$ by 
$\E{Y\mathbf{1}_{Y>R}}\leq R^{1-p\pr{\alpha}}
 \sup_{t>0} t^{p\pr{\alpha}}
    \PP\pr{Y>t }$. This ends the proof of Proposition~\ref{prop:neg_terme_biais}.
\end{proof}
\subsection{Integrability results}
The main results assume uniform integrability of the summands of the $U$-statistics. This property is also required for the $U$-statistics arising in Hoeffding's decomposition. The next proposition shows that no supplementary assumption is required. In this section, $\pr{X_i}_{i\in \Z}$ is a strictly stationary sequence and $h\colon S^m\to \R$ is a symmetric and measurable function . 

\begin{Proposition}\label{prop:UI_carre}
Let $h\colon S^m\to \R$ be a measurable function and let $\pr{X_i}_{i\in \Z}$ be a strictly stationary sequence.
Suppose that for each $k\in\intent{1,m}$ the family $\ens{h\pr{X_{\gri},X_1^*,\dots,X_{m-k}^*}^2,\gri\in \inc^k} $
 is uniformly integrable.
  Let $\pi_{k,m}\pr{h}$ be defined by 
\eqref{eq:definition_de_pi_km}. Then for each $k\in\intent{1,m}$, 
the family $\pr{\pr{\pr{\pi_{k,m}\pr{h}}\pr{X_{\gri}}}^2,\gri\in \inc^m}$ is uniformly integrable.
\end{Proposition}
\begin{Proposition}\label{prop:unif_bound_H2_log}
For $a\geq 0$, let $\varphi_a\colon \R \to\R_+$ defined by $\varphi_a\pr{x}\mapsto x^2\pr{\log \pr{1+\abs{x}}}^a$. Let $h\colon S^m\to \R$ be a measurable function and let $\pr{X_i}_{i\in \Z}$ be a strictly stationary sequence.
Suppose that for each $k\in\intent{1,m}$ and each 
$\gri:=\pr{i_1,\dots,i_k}\in \inc^m$, 
\[
\E{\varphi_a\pr{h\pr{X_{i_1},\dots,X_{i_k},X_1^*,\dots,X_{m-k}^*}}}<+\infty,
\]
where the vector $\pr{X_1^*,\dots,X_{m-k}^*}$ is  independent of $\pr{X_i}_{i\in \Z}$ and 
consists of independent random variables having the
same distribution as  (when $m=k$, the previous condition reads simply as $\E{\varphi_a\pr{h\pr{X_{i_1},\dots,X_{i_m}}}}<+\infty$). Suppose moreover that  $\sup_{\gri\in I^m_\infty }\E{\varphi_a\pr{h\pr{X_{\gri}}} }<+\infty$. Let $\pi_{k,m}\pr{h}$ be defined by 
\eqref{eq:definition_de_pi_km}. Then for each $k\in\intent{1,m}$, 
\begin{equation*}
 \sup_{\gri\in \inc^k}  \E{\varphi_a\pr{\pr{\pi_{k,m}\pr{h}}\pr{X_{\gri}}} } <+\infty
\end{equation*}
\end{Proposition}
 
\begin{Proposition}\label{prop:tail_Holder}
Let $h\colon S^m\to \R$ be a measurable function and let $\pr{X_i}_{i\in \Z}$ be a strictly stationary sequence. Suppose that 
\begin{equation*}
    \lim_{t\to +\infty}\max_{1\leq k\leq m}
    \sup_{\gri\in\inc^k} t^p\PP\pr{\abs{h\pr{X_{i_1},\dots,X_{i_k},X_1^*,\dots,X_{m-k}^*}}>t  }=0.
\end{equation*}
Then for each $k\in\intent{1,m}$, 
\begin{equation*}
   \lim_{t\to +\infty} t^p \sup_{\gri\in I^k_\infty}  \PP\pr{\abs{\pi_{k,m}\pr{h}\pr{X_{\gri}}}>t}=0. 
\end{equation*}
\end{Proposition}
\begin{proof}[Proof of Proposition ~\ref{prop:UI_carre}]
Notice that the random variable $\pi_{k,m}\pr{h}\pr{X_{\gri}}$ is a finite sum of random variables of the form $\E{H_{\gri,k}\mid X_{i_1},\dots,X_{i_k},X_1^*,\dots,X_j^*}$. Therefore, it suffices to prove   that for any $\sigma$-algebra $\Aca$ and each $k\in \intent{1,m-1}$,  
\begin{equation*}
   \ens{\pr{\E{h\pr{X_{\gri},X_1^*,\dots,X_{m-k}^*}\mid \Aca}}^2,\gri\in \inc^k} \mbox{ is uniformly integrable},
\end{equation*}
 To do so, we use Jensen's inequality and the standard fact that if $\ens{\abs{Y_j},j\in J}$ is uniformly integrable, so is $\ens{\E{\abs{Y_j}\mid \Aca},j\in J}$, for example using de la Vall\'ee Poussin theorem.
\end{proof} 
\begin{proof}[Proof of Proposition~ \ref{prop:unif_bound_H2_log}]
Similarly as before, it suffices to prove that for any $\sigma$-algebra $\Aca$ and each $k\in \intent{1,m-1}$, 
\begin{equation*}
   \sup_{\gri\in \inc^k} \E{\varphi_a\pr{ \E{   h\pr{X_{i_1},\dots,X_{i_k},X_1^*,\dots,X_{m-k}^*}\mid \Aca}}}<+\infty.
\end{equation*}
This follows from convexity of $\varphi_a$ and Jensen's inequality.
\end{proof}
\begin{proof}[Proof of Proposition~ \ref{prop:tail_Holder}]
Similarly as before, it suffices to prove that 
for each $\sigma$-algebra $\Aca$, 
\begin{equation}\label{eq:cond_exp_tail}
   \lim_{t\to +\infty} t^p \sup_{\gri\in \inc^k}  \PP\pr{ \E{\abs{h\pr{X_{i_1},\dots,X_{i_k},X_1^*,\dots,X_{m-k}^*}} \mid \Aca }>t}=0.
\end{equation}
For a non-negative random variable $Y$, 
\begin{align*}
    t^p\PP\pr{\E{Y\mid \Aca}>t}&\leq t^{p-1}
    \E{\E{Y\mid \Aca}\mathbf{1}_{\E{Y\mid \Aca}>t}}\\
    &=t^{p-1}
    \E{Y\mathbf{1}_{\E{Y\mid \Aca}>t}}\\
    &=t^{p-1}\int_0^{+\infty}\PP\pr{Y>s,\E{Y\mid \Aca}>t}ds\\
    &\leq \frac{t^p}2\PP\pr{ \E{Y\mid \Aca}>t}
    +t^{p-1}\int_{t/2}^{+\infty}\PP\pr{Y>s}ds\\
    &\leq \frac{t^p}2\PP\pr{ \E{Y\mid \Aca}>t}
    + \frac{2^p}{p-1}\sup_{s\geq t/2}s^p\PP\pr{Y>s},
\end{align*}
which gives 
\[
  t^p\PP\pr{\E{Y\mid \Aca}>t}\leq \frac{2^{p+1}}{p-1}\sup_{s\geq t/2}s^p\PP\pr{Y>s},
\]
from which \eqref{eq:cond_exp_tail} follows.
\end{proof}
\section{Proof of the results}\label{sec:proofs}

\subsection{Proof of Theorem~\ref{thm:TLC}}
 
In view of the decomposition \eqref{eq:Hoeffding}, it suffices to prove that 
\begin{equation}\label{eq:TLC_h1}
    \frac 1{\sqrt{n}}\sum_{i=1}^n\pr{h_{1,m}\pr{X_i}-\E{h_{1,m}\pr{X_i}}}\to m\sigma\Nca\mbox{ in distribution and}
\end{equation}
\begin{equation}\label{eq:TLC_partie_degeneree}
    \forall c\in\intent{2,m}, 
    \frac 1{n^{c-1/2}}\sum_{\gri\in\inc^c_n}
    \pr{\pr{\pi_{c,m}\pr{h}}\pr{X_{\gri}}-\E{\pr{\pi_{c,m}\pr{h}}\pr{X_{\gri}}} }\to 0\mbox{ in probability}.
\end{equation}
By \cite{zbMATH02233802}, Theorem 2, condition \eqref{eq:TLC_DMR} implies the convergence \eqref{eq:TLC_h1}. We decompose the proof of 
\eqref{eq:TLC_partie_degeneree} into the following steps:
\begin{equation}\label{eq:TLC_partie_degeneree_variance}
     \forall c\in\intent{2,m}, \frac 1{n^{c-1/2}}\sum_{\gri\in\inc^c_n}
    \pr{\pi_{c,m}\pr{h}}\pr{X_{\gri}} \to 0 \mbox{ in probability},
\end{equation}
\begin{equation}\label{eq:TLC_partie_degeneree_terme_de_biais}
    \forall c\in\intent{2,m}, \frac 1{n^{c-1/2}}\sum_{\gri\in\inc^c_n}\E{\pr{\pi_{c,m}\pr{h}}\pr{X_{\gri}}} \to 0.
\end{equation}
Let us show \eqref{eq:TLC_partie_degeneree_variance}. Applying 
\eqref{eq:inegalite_accroissements_h_canonique} with $a=0$, $n=N$, $m$ replaced by $c$ and $t=\eps$, we derive that \begin{align*}
\PP\pr{\frac 1{n^{c-1/2}}\abs{\sum_{\gri\in\inc^c_n}
    h_{c,m}\pr{X_{\gri}}}>\eps n^{c-1/2}}
    \leq R^2\frac{\kappa_m}{\eps^2}\frac 1n\pr{1+\sum_{k=1}^nk\beta\pr{k}}\\+\frac{\kappa_m}{\eps}
\sqrt{n}\sup_{\gri\in\inc^c_m}\E{\abs{h_{c,m}\pr{X_{\gri}} }
\mathbf{1}_{\abs{h_{c,m}\pr{X_{\gri}} }>R}}.
\end{align*}
Define 
\begin{equation}\label{eq:serie_coeff_beta}
    B:=1+\sum_{k=1}^{+\infty}k\beta\pr{k}
\end{equation}
For a fixed $\eta$, the choice \[
R=\eps \frac{\sqrt{n\eta}}{\sqrt{\kappa_m B}}
\]
combined with the elementary bound $\E{Y\mathbf{1}_{Y>R}}\leq R^{-1} \E{Y^2\mathbf{1}_{Y>R}}$ gives
\begin{multline}
    \PP\pr{\frac 1{n^{c-1/2}}\abs{\sum_{\gri\in\inc^c_n}
    h_{c,m}\pr{X_{\gri}}}>\eps n^{c-1/2}}
    \leq \eta \\ 
    +\frac{\kappa_m^{3/2}B}{\eps^2\sqrt{\eta}} 
    \sup_{\gri\in\inc^c_m}\E{\pr{h_{c,m}\pr{X_{\gri}} }^2
\mathbf{1}_{\abs{h_{c,m}\pr{X_{\gri}} }>\eps \frac{\sqrt{n\eta}}{\sqrt{\kappa_mB}}}}
\end{multline}
and we conclude by Proposition~\ref{prop:UI_carre}. Finally, \eqref{eq:TLC_partie_degeneree_terme_de_biais} follows from an application of Proposition~\ref{prop:neg_terme_biais}.

\subsection{Proof of Theorem~\ref{thm:TLCF_C}} 
In view of the decomposition \eqref{eq:Hoeffding}, the partial sum process can be written as 
\begin{equation}\label{eq:decomposition_pour_TLCF}
  \frac{\sqrt{n}}{\binom nm}\pr{W_{m,n}\pr{h,t}-\E{W_{m,n}\pr{h,t}} } :=Y_{n}\pr{t}+R_n\pr{t},
\end{equation}
where 
\begin{equation*}
Y_n\pr{t}=\begin{cases} 
m\frac{\binom km }{k}\frac{\sqrt{n}}{\binom nm}\pr{U_{1,k}\pr{\pi_{1,m}\pr{h}}
-\E{U_{1,k}\pr{\pi_{1,m}\pr{h}}}}&\mbox{ if }t=k/n,\\
\mbox{linear interpolation }& \mbox{ on }[k/n,\pr{k+1}/n,
\end{cases}
\end{equation*}
and $R_n\pr{t}$ is defined by $\sum_{c=2}^m R_{n,c}\pr{t}$, 
where 
\begin{equation*}
R_{n,c}\pr{t}=\begin{cases} \frac{\sqrt{n}}{\binom nm}\frac{\binom km \binom mc}{\binom kc}\pr{U_{c,k}\pr{\pi_{c,m}\pr{h}}
-\E{U_{c,k}\pr{\pi_{c,m}\pr{h}}}}&\mbox{ if }t=k/n,\\
\mbox{linear interpolation }& \mbox{ on }[k/n,\pr{k+1}/n.
\end{cases}
\end{equation*}
As before, the convergence is carried by the term corresponding to the first one in Hoeffding's decomposition and the contribution of $R_{n}\pr{t}$ will be shown to be negligible. Due to the presence of the terms $\frac{m}{k}\binom{k}{m}$,
$Y_n$ is not exactly a partial sum process based on a strictly stationary sequence. We define
\begin{equation}\label{eq:def_de_Y'n}
    Y'_n\pr{t}=\begin{cases} mt^{m-1}n^{-1/2}\pr{U_{1,k}\pr{\pi_{1,m}\pr{h}}
-\E{U_{1,k}\pr{\pi_{1,m}\pr{h}}}} &\mbox{ if }t=k/n,\\
\mbox{linear interpolation }& \mbox{ on }[k/n,\pr{k+1}/n.
\end{cases}
\end{equation}
By \cite{zbMATH02233802}, Theorem 2, condition \eqref{eq:TLC_DMR} implies the convergence in $C[0,1]$ of the sequence of processes 
\begin{equation} \label{eq:def_de_Y''n}
    Y''_n\pr{t}=\begin{cases} n^{-1/2}\pr{U_{1,k}\pr{\pi_{1,m}\pr{h}}
-\E{U_{1,k}\pr{\pi_{1,m}\pr{h}}}}  &\mbox{ if }t=k/n,\\
\mbox{linear interpolation }& \mbox{ on }[k/n,\pr{k+1}/n
\end{cases}
\end{equation}
to the limit that appears in \eqref{eq:TLCF_cont}.
Therefore, in order to prove the convergence of $\pr{Y_n\pr{t}}_{0\leq t\leq 1}$ to the limiting process that appears in \eqref{eq:TLCF_cont}, it suffices to prove that 
\begin{equation*}
    \lim_{n\to +\infty}\max_{m\leq k\leq n}\abs{\frac{\binom km }{k}\frac{ n }{\binom nm}-\pr{\frac kn}^{m-1}}=0.
\end{equation*}
This follows from the following inequalities, valid for $k\in \intent{m,n}$:
\begin{align*}
    \abs{\frac{\binom km }{k}\frac{ n }{\binom nm}-\pr{\frac kn}^{m-1}}&=\abs{\frac{\pr{k-1}!}{\pr{k-m}!}
    \frac{\pr{n-m}!}{\pr{n-1}!}-\pr{\frac kn}^{m-1}}\\
    &\leq \abs{\frac{\pr{k-1}!}{\pr{k-m}!}
    \pr{\frac{\pr{n-m}!}{\pr{n-1}! }-\frac{1}{n^{m-1}}}}+
   \frac{1}{n^{m-1}} \abs{\frac{\pr{k-1}!}{\pr{k-m}!}
    -  k  ^{m-1}}\\
    &\leq \abs{ \frac{\pr{n-m}!}{\pr{n-1}! }-\frac{1}{n^{m-1}} }+\frac{1}{n^{m-1}}
    \pr{k^{m-1}-\pr{k-1}^{m-1} }\\
    &\leq \abs{ \frac{\pr{n-m}!}{\pr{n-1}! }-\frac{1}{n^{m-1}} }+\frac{mn^{m-2}}{n^{m-1}}.
\end{align*}
It remains to check that the contribution of the processes $R_{n,c}\pr{t}$ is negligible. Bounding 
$\binom km/\binom kc$ by $\kappa_{m,c }n^{m-c}$, it suffices to show that 
\begin{equation*}
\frac 1 {n^{c-1/2}}    \max_{m\leq k\leq  n}
\abs{\pr{U_{c,k}\pr{\pi_{c,m}\pr{h}}
-\E{U_{c,k}\pr{\pi_{c,m}\pr{h}}}}}\to 0\mbox{ in probability}.
\end{equation*}
Applying Proposition~\ref{prop:neg_terme_biais}, 
it suffices to show that 
\begin{equation*}
\frac 1 {n^{c-1/2}}    \max_{m\leq k\leq  n}
\abs{\pr{U_{c,k}\pr{\pi_{c,m}\pr{h}}
 }}\to 0\mbox{ in probability}.
\end{equation*}
Let $\eps>0$ be fixed. 
Applying \eqref{eq:inegalite_max_h_canonique} with $\widetilde{m}=c$, $\widetilde{h}=\pi_{c,m}\pr{h}$ 
and $t=\eps$ gives 
\begin{multline}
    \PP\pr{\frac 1 {n^{c-1/2}}    \max_{m\leq k\leq  n}
\abs{\pr{U_{c,k}\pr{\pi_{c,m}\pr{h}}
 }}>\eps }\leq \frac{C_c}{\eps^2}       \pr{\log n}^2\frac{R^2}{n}  B\\+\frac{C_c}{\eps}
\sqrt{n}\sup_{\gri\in\inc^c_m}\E{\abs{h_{c,m}\pr{X_{\gri}} }
\mathbf{1}_{\abs{h_{c,m}\pr{X_{\gri}} }>R}},
\end{multline}
where $B$ is defined as in \eqref{eq:serie_coeff_beta}.
Letting $K:= C_cB $ allows to rewrite the previous bound under the form 
\begin{multline}
 \PP\pr{\frac 1 {n^{c-1/2}}    \max_{m\leq k\leq  n}
\abs{\pr{U_{c,k}\pr{\pi_{c,m}\pr{h}}
 }}>\eps }\\
 \leq K\pr{ R^2\frac{\log n}{\eps^2n}+\frac{\sqrt{n}}{\eps}\sup_{\gri\in\inc^c_m}\E{\abs{h_{c,m}\pr{X_{\gri}} }
\mathbf{1}_{\abs{h_{c,m}\pr{X_{\gri}} }>R}}    }  . 
\end{multline}
For a fixed $\eta>0$, take $R^2=\eta\eps^2 n/\pr{K\log n}$, so that the bound becomes 
\begin{multline}\label{eq:bound_WIP_C01_deg}
 \PP\pr{\frac 1 {n^{c-1/2}}    \max_{m\leq k\leq  n}
\abs{\pr{U_{c,k}\pr{\pi_{c,m}\pr{h}}
 }}>\eps }\\
 \leq \eta+ K  \frac{\sqrt{n}}{\eps}\sup_{\gri\in\inc^c_m}\E{\abs{h_{c,m}\pr{X_{\gri}} }
\mathbf{1}_{\abs{h_{c,m}\pr{X_{\gri}} }>\sqrt{\frac{\eta}K}\eps
\frac{\sqrt{n}}{\sqrt{\log n}}     }}      . 
\end{multline}
Notice that for a positive random variable $Y$ and for $A,a>0$,  
\[
\E{Y\mathbf{1}_{Y>A}}\leq \frac 1{A\pr{\log\pr{1+A}}^a}\E{Y^2 \pr{\log\pr{Y+1}}^a}. 
\]
Plugging this bound into \eqref{eq:bound_WIP_C01_deg} gives 
\begin{multline}\label{eq:bound_WIP_C01_deg_bis}
 \PP\pr{\frac 1 {n^{c-1/2}}    \max_{m\leq k\leq  n}
\abs{ U_{c,k}\pr{\pi_{c,m}\pr{h}}
  }>\eps }\\
 \leq \eta+ \kappa\pr{a,\eps,\eta,K} \frac{\sqrt{\log n}}{\pr{\log\pr{1+\sqrt{\frac{n}{\log n}}}}^a}\sup_{\gri\in\inc^c_m}\E{\pr{h_{c,m}\pr{X_{\gri}} }^2 \pr{\log\pr{1+\abs{ h_{c,m}\pr{X_{\gri}  }}}}^a}   ,
\end{multline}
where $\kappa\pr{a,\eps,\eta,K}$ is independent of $n$. Since $a>1/2$, 
Proposition~\ref{prop:unif_bound_H2_log} concludes the proof of Theorem~\ref{thm:TLCF_C}.
\subsection{Proof of Theorem~\ref{thm:TLC_Holder}} 
 
 Here again we use the decomposition \eqref{eq:decomposition_pour_TLCF}.  
 Using the same arguments as in the proof of Theorem~\ref{thm:TLCF_C} combined with the fact that 
 \begin{equation}\label{eq:diff_processus_avec_ou_sans_t_Holder}
    \lim_{n\to +\infty}n^\alpha\max_{m\leq k\leq n}\abs{\frac{\binom km }{k}\frac{ n }{\binom nm}-\pr{\frac kn}^{m-1}}=0,
\end{equation}
we are reduced to show that $\pr{Y''_n\pr{t}}_{0\leq t\leq 1}$ defined as \eqref{eq:def_de_Y''n} converges to $\pr{\sigma B_t}_{0\leq t\leq 1}$ and that 
the H\"older norm of the processes $\pr{R_{n,c}\pr{t}}_{0\leq t\leq 1}$ goes to $0$ in probability as $n$ goes to infinity.
The first part follows from Corollary~2.2 in \cite{zbMATH06705457} and \eqref{eq:DMR_condition_Holder}.
Using again \eqref{eq:diff_processus_avec_ou_sans_t_Holder}, it suffices to show that $\pr{R'_{n,c}\pr{t}}_{0\leq t\leq 1}$ goes to $0$ in probability as $n$ goes to infinity where 
\begin{equation*}
R'_{n,c}\pr{t}=\begin{cases}  \frac 1{n^{c/2}}\pr{U_{c,k}\pr{\pi_{c,m}\pr{h}}
-\E{U_{c,k}\pr{\pi_{c,m}\pr{h}}}}&\mbox{ if }t=k/n,\\
\mbox{linear interpolation }& \mbox{ on }[k/n,\pr{k+1}/n.
\end{cases}
\end{equation*}
We have already seen the convergence to $0$ of the finite dimensional distributions. It remains to check its tightness, which will be done thanks to  Proposition~1.1 in \cite{zbMATH07420266}. This reduces the proof to show that for each positive $\eps$, 
\begin{equation*}
    \lim_{J\to +\infty}\limsup_{n\to +\infty}
    \sum_{j=J}^{\ent{\log_2 n}}
    \sum_{k=0}^{2^{j}-1}
    \PP\pr{\abs{U_{c,\ent{n\pr{k+1}2^{-j}}  }\pr{\pi_{c,m}\pr{h}}-U_{c,\ent{nk2^{-j}}  }\pr{\pi_{c,m}\pr{h}} }> n^{c-1/2}2^{-j\alpha}\eps}=0.
\end{equation*}
Applying inequality \eqref{eq:inegalite_accroissements_h_canonique} for fixed $J\geq 1$, $j\in \intent{J,\ent{\log_2 n}}$ and $k\in \intent{0,2^j-1}$ with $m=c$, $\widetilde{h}=\pi_{c,m}\pr{h}$, $a=\ent{nk2^{-j}}$, $N=\ent{n\pr{k+1}2^{-j}}-\ent{nk2^{-j}}$ and $t=2^{-j\alpha}\eps$ gives 
\begin{align*}
  & \sum_{j=J}^{\ent{\log_2 n}}
    \sum_{k=0}^{2^{j}-1}
    \PP\pr{\abs{U_{c,\ent{n\pr{k+1}2^{-j}}  }\pr{\pi_{c,m}\pr{h}}-U_{c,\ent{nk2^{-j}}  }\pr{\pi_{c,m}\pr{h}} }> n^{c-1/2}2^{-j\alpha}\eps}\\
    &\leq C\pr{c,\eps,\beta}\sum_{j=J}^{\ent{\log_2 n}}
    \sum_{k=0}^{2^{j}-1}\pr{ R^2 \frac{n2^{-j}\pr{nk2^{-j}}^{2c-3}}{\pr{n^{c-1/2}2^{-j\alpha}}^2}+
    \frac{n2^{-j}\pr{nk2^{-j}}^{c-1}}{ n^{c-1/2}2^{-j\alpha} }
     \tau_c\pr{R}}\\
    &\leq  C\pr{c,\eps,\beta}\sum_{j=J}^{\ent{\log_2 n}}\pr{\frac{R^2}n 2^{j\alpha}+ n^{1/2}
    2^{j\alpha }\tau_c\pr{R}}\\
    &\leq C\pr{c,\eps,\beta}\pr{R^2
    n^{ -1+2\alpha}+n^{1/2+\alpha}\tau_c\pr{R} },
\end{align*}
where 
\[
\tau_c\pr{R}:=\sup_{\gri\in\inc^c}
    \E{\abs{\pi_{c,m}\pr{h}\pr{X_{\gri}}}\mathbf{1}_{\abs{\pi_{c,m}\pr{h}\pr{X_{\gri}}}>R} } .
\]
For a fixed positive $\eta$, we choose $R=\eta n^{1/2-\alpha}$ and conclude by \eqref{eq:tail_condition_Holder}.
\providecommand{\bysame}{\leavevmode\hbox to3em{\hrulefill}\thinspace}
\providecommand{\MR}{\relax\ifhmode\unskip\space\fi MR }
\providecommand{\MRhref}[2]{%
  \href{http://www.ams.org/mathscinet-getitem?mr=#1}{#2}
}
\providecommand{\href}[2]{#2}

\end{document}